\documentclass{article}

\newtheorem{theorem}{Theorem}

\newtheorem{lemma}{Lemma}
\newtheorem{remark}{Remark}
\newenvironment{proof}[1][Proof]{\noindent\textbf{#1.} }{\ \rule{0.5em}{0.5em}}
\input{tcilatex}

\begin{document}

\title{The foci and rotation angle of an ellipse, $E_{0}$, as a function of
the coefficients of an equation of $E_{0}$ }
\author{Alan Horwitz}
\date{5/27/17}
\maketitle

\begin{abstract}
First, we give a formula for the foci of an ellipse, $E_{0}$, as a function
of the coefficients of an equation of $E_{0}$(see Theorem \ref{T2}). To
prove Theorem \ref{T2}, we use two interesting formulas proven in \cite{B}
and in \cite{S}. Our second result(see Theorem \ref{T3}), is a more precise
formula for the rotation angle, $\theta $, of $E_{0}$, as a function of the
coefficients of an equation of $E_{0}$.
\end{abstract}

\section{Introduction}

The purpose of this note is two fold. First, we give a formula for the foci
of an ellipse, $E_{0}$, as a function of the coefficients of an equation of $%
E_{0}$(see Theorem \ref{T2}); To prove Theorem \ref{T2}, we use two
interesting formulas proven in \cite{B} and in \cite{S}. The main result in 
\cite{B} expresses the foci of an ellipse, $E_{0}$, as a function of the
coefficients of an equation of $E_{0}$, but also requires knowing the length
of the major axis of $E_{0}$; We expand on that formula a little and give
the proof here(see Theorem \ref{T1}). A formula in \cite{S} yields the
length of the major axis of $E_{0}\ $as a function of the coefficients(see
Lemma \ref{L1}). Theorem \ref{T1} and Lemma \ref{L1} then yield Theorem \ref%
{T2}.

There are various ways to define the rotation angle, $\theta $, of a
non--circular ellipse, $E_{0}$; Below we define $\theta $ to be the
counterclockwise angle of rotation to the major axis of $E_{0}$ from the
line thru the center of $E_{0}$ and parallel to the $x$\ axis, with $0\leq
\theta <\pi $; No matter how one defines $\theta $, it is always true that $%
\cot (2\theta )=\dfrac{A-C}{B}$; But what about a formula for $\theta $
itself ? Our second result(see Theorem \ref{T3}), is a more precise formula
for the rotation angle, $\theta $, of $E_{0}$, as a function of the
coefficients of an equation of $E_{0}$. The latter formula was submitted as
a correction for the previous formula for the rotation angle given in \cite%
{M}. The formula given here now appears in \cite{M} in a slightly different
form. The proof of Theorem \ref{T3} then follows easily from the proof of
Theorem \ref{T1}.

While the formulas given in this note are undoubedtly known, and there are
other ways of proving them, we found it interesting to use and highlight the
results in \cite{B} and in \cite{S}.

\section{Foci as a Function of the Coefficients}

Throughout, for a given ellipse, $E_{0}$, which is not a circle, we let $%
\theta $\ denote the counterclockwise angle of rotation to the major axis of 
$E_{0}$ from the line thru the center of $E_{0}$ and parallel to the $x$\
axis, with $0\leq \theta <\pi $; We let $(x_{0},y_{0})=$ center of $E_{0},a=$
length of semi--major and $b=$ length of semi--minor axes of $E_{0}$,
respectively. Finally, we let $F_{2}=\left( x_{c},y_{c}\right) $ denote the
rightmost focus of $E_{0}$(if $\theta =\dfrac{\pi }{2}$, we let $F_{2}$
denote the uppermost focus). Knowing $F_{2}$ easily yields the other focus, $%
F_{1}=\left( 2x_{0}-x_{c},2y_{0}-y_{c}\right) $;

We now state an extension, and give a detailed proof, of the result in \cite%
{B}. (i) gives the equation of an ellipse, $E_{0}$, given the foci of $E_{0}$%
, while (ii) gives the foci of $E_{0}$ given the equation of $E_{0}$. In
each case, one must also know the length of the semi--major axis of $E_{0}$.

\begin{theorem}
\label{T1} Let $E_{0}$ be an ellipse which is not a circle, let $%
F_{2}=\left( x_{c},y_{c}\right) $ be the rightmost focus of $E_{0}$, and let 
$(x_{0},y_{0})$ be the center of $E_{0}$.

(i) Then the equation of $E_{0}$ can be written in the form $A\left(
x-x_{0}\right) ^{2}+B\left( x-x_{0}\right) \left( y-y_{0}\right) +C\left(
y-y_{0}\right) ^{2}-a^{2}b^{2}=0$, where $%
A=a^{2}-(x_{c}-x_{0})^{2},B=-2(x_{c}-x_{0})(y_{c}-y_{0})$, and $%
C=a^{2}-(y_{c}-y_{0})^{2}$.

(ii) If the equation of $E_{0}$ is written in the form $A\left(
x-x_{0}\right) ^{2}+B\left( x-x_{0}\right) \left( y-y_{0}\right) +C\left(
y-y_{0}\right) ^{2}-a^{2}b^{2}=0$, where $A,C>0$, then 
\begin{eqnarray}
x_{c} &=&x_{0}+\sqrt{a^{2}-A}{\large ,}  \label{foci1} \\
y_{c} &=&y_{0}-(\limfunc{sgn}B)\sqrt{a^{2}-C}{\large )}\text{ if }B\neq 0%
\text{.}  \nonumber
\end{eqnarray}%
In addition, if $0\leq \theta <\dfrac{\pi }{2}$, then $B<0$, while if $%
\dfrac{\pi }{2}<\theta <\pi $, then $B>0$. Finally, 
\begin{eqnarray}
x_{c} &=&x_{0}+\dfrac{1}{2}{\large (}1-\limfunc{sgn}(A-C){\large )}\sqrt{%
a^{2}-A}{\large ,}  \label{foci2} \\
y_{c} &=&y_{0}+\dfrac{1}{2}{\large (}1+\limfunc{sgn}(A-C){\large )}\sqrt{%
a^{2}-C}{\large )}\text{ if }B=0\text{.}  \nonumber
\end{eqnarray}%
\newline
\end{theorem}

\begin{proof}
It is clear that we may assume that $x_{0}=y_{0}=0$, so that the equation of 
$E_{0}$ has the form 
\begin{equation}
Ax^{2}+Bxy+Cy^{2}+G=0\text{.}  \label{1}
\end{equation}%
The implicit assumption in \cite{B} is that $0\leq \theta <\dfrac{\pi }{2}$;
We outline the proof in the case when $\dfrac{\pi }{2}\leq \theta <\pi $ as
well. If $0\leq \theta <\dfrac{\pi }{2}$, then $F_{2}$ lies in quadrant 1,
while if $\dfrac{\pi }{2}\leq \theta <\pi $, then $F_{2}$ lies in quadrant
4; Letting $c=\sqrt{a^{2}-b^{2}}$, we then have 
\begin{equation}
\left\{ 
\begin{array}{ll}
x_{c}=c\cos \theta ,y_{c}=c\sin \theta & \text{if }0\leq \theta \leq \dfrac{%
\pi }{2} \\ 
x_{c}=-c\cos \theta ,y_{c}=-c\sin \theta & \text{if }\dfrac{\pi }{2}<\theta
<\pi%
\end{array}%
\right. \text{.}  \label{4}
\end{equation}

Recall that if $\theta =\dfrac{\pi }{2}$, then $F_{2}$ is the uppermost
focus. Proceeding as in \cite{B}(we include the details here for
completeness), we have:

$\sqrt{(x-x_{c})^{2}+(y-y_{c})^{2}}+\sqrt{(x+x_{c})^{2}+(y+y_{c})^{2}}=2a$,
which implies that

$(x+x_{c})^{2}+(y+y_{c})^{2}=4a^{2}-4a\sqrt{(x-x_{c})^{2}+(y-y_{c})^{2}}%
+(x-x_{c})^{2}+(y-y_{c})^{2}$, and so

$a^{2}{\large (}(x-x_{c})^{2}+(y-y_{c})^{2}{\large )}%
=a^{4}-2a^{2}(xx_{c}+yy_{c})+(xx_{c}+yy_{c})^{2}$; Some simplification yields

$(a^{2}-x_{c}^{2})x^{2}-\allowbreak
2x_{c}y_{c}xy+(a^{2}-y_{c}^{2})y^{2}+a^{2}(c^{2}-a^{2})=0$, and using $%
b^{2}=a^{2}-c^{2}$ gives 
\begin{equation}
(a^{2}-x_{c}^{2})x^{2}-\allowbreak
2x_{c}y_{c}xy+(a^{2}-y_{c}^{2})y^{2}-a^{2}b^{2}=0\text{.}  \label{2}
\end{equation}

Matching (\ref{1}) with (\ref{2}) yields 
\begin{equation}
A=a^{2}-x_{c}^{2},B=-\allowbreak 2x_{c}y_{c},C=a^{2}-y_{c}^{2}\text{,}
\label{3}
\end{equation}%
and $G=-a^{2}b^{2}$, which proves (i). To prove (ii): Note that $\left\vert
x_{c}\right\vert =\sqrt{a^{2}-A}$ and $\left\vert y_{c}\right\vert =\sqrt{%
a^{2}-C}$.

\textbf{Case 1: }$B\neq 0$

Then $x_{c}\neq 0\neq y_{c}$ by (\ref{3}), which implies that $0\neq \theta
\neq \dfrac{\pi }{2}$ by (\ref{4}); If $0<\theta <\dfrac{\pi }{2}$, then $%
x_{c}>0$ and $y_{c}>0$ by (\ref{4}), which implies that $B<0,x_{c}=\sqrt{%
a^{2}-A}$, and $y_{c}=\sqrt{a^{2}-C}$; If $\dfrac{\pi }{2}<\theta <\pi $,
then $x_{c}>0$ and $y_{c}<0$ by (\ref{4}), which implies that $B>0,x_{c}=%
\sqrt{a^{2}-A}$, and $y_{c}=-\sqrt{a^{2}-C}$; That proves (\ref{foci1}).

\textbf{Case 2: }$B=0$

Then $x_{c}=0$ or $y_{c}=0$ by (\ref{3}); If $A<C$, then by (\ref{3}) again, 
$a^{2}-x_{c}^{2}<a^{2}-y_{c}^{2}$, which implies that $y_{c}^{2}<x_{c}^{2}$,
and so $y_{c}=0$; Thus $\theta =0$, which implies that $x_{c}=c>0$ by (\ref%
{4}) and so $x_{c}=\sqrt{a^{2}-A}$; If $A>C$, then $%
a^{2}-x_{c}^{2}>a^{2}-y_{c}^{2}$, which implies that $y_{c}^{2}>x_{c}^{2}$,
and so $x_{c}=0$; Thus $\theta =\dfrac{\pi }{2}$, which implies that$%
y_{c}=c>0$ by (\ref{4}) and so $y_{c}=\sqrt{a^{2}-C}$; That proves (\ref%
{foci2}).

For the following two lemmas, we let $\Delta =4AC-B^{2}$ and $\delta
=CD^{2}+AE^{2}-BDE-F\Delta $. The following result can be found in \cite{S}.
\end{proof}

\begin{lemma}
\label{L1}Suppose that $E_{0}$ is an ellipse with equation $%
Ax^{2}+Bxy+Cy^{2}+Dx+Ey+F=0$; Let $a$ and $b$ denote the lengths of the
semi--major and semi--minor axes, respectively, of $E_{0}$, and let $\mu =%
\dfrac{4\delta }{\Delta ^{2}}$. Then 
\begin{eqnarray}
a^{2} &=&\mu \dfrac{A+C+\sqrt{(A-C)^{2}+B^{2}}}{2}  \label{absq} \\
b^{2} &=&\mu \dfrac{A+C-\sqrt{(A-C)^{2}+B^{2}}}{2}\text{.}  \nonumber
\end{eqnarray}
\end{lemma}

We state the following useful general lemma about equations of ellipses. The
second condition ensures that the conic is non--degenerate, while the first
condition ensures that the conic is an ellipse.

\begin{lemma}
\label{L2}The equation $Ax^{2}+Bxy+Cy^{2}+Dx+Ey+F=0$, with $A,C>0$, is the
equation of an ellipse if and only if $\Delta >0$ and $\delta >0$.
\end{lemma}

Using Lemma \ref{L1}, we are now able to give a formula for the foci of $%
E_{0}$, given an equation of $E_{0}$, \textbf{without} knowing the length of
the semi--major axis of $E_{0}$, as with Theorem \ref{T1}.

\begin{theorem}
\label{T2}Let $E_{0}$ be an ellipse which is not a circle, and let $r=\sqrt{%
(A-C)^{2}+B^{2}}$; Let $F_{2}=\left( x_{c},y_{c}\right) $ be the rightmost
focus of $E_{0}$ and let $(x_{0},y_{0})$ be the center of $E_{0}$. If the
equation of $E_{0}$ is written in the form $A\left( x-x_{0}\right)
^{2}+B\left( x-x_{0}\right) \left( y-y_{0}\right) +C\left( y-y_{0}\right)
^{2}-a^{2}b^{2}=0$, where $A,C>0$, then 
\begin{eqnarray*}
x_{c} &=&x_{0}+\sqrt{(r+C-A)/2}, \\
y_{c} &=&y_{0}-(\limfunc{sgn}B)\sqrt{(r+A-C)/2}{\large )}\text{ if }B\neq 0%
\text{,}
\end{eqnarray*}%
\begin{eqnarray*}
x_{c} &=&x_{0}+\dfrac{1}{2}{\large (}1-\limfunc{sgn}(A-C){\large )}\sqrt{%
(r+C-A)/2}, \\
y_{c} &=&y_{0}+\dfrac{1}{2}{\large (}1+\limfunc{sgn}(A-C){\large )}\sqrt{%
(r+A-C)/2}{\large \ }\text{if }B=0\text{.}
\end{eqnarray*}
\end{theorem}

\begin{remark}
\label{R1}To use Theorem \ref{T2}, one must first rewrite the equation of $%
E_{0}$ so that it has the form given in Theorem \ref{T2}. First one writes
the equation of $E_{0}$ in the form $A\left( x-x_{0}\right) ^{2}+B\left(
x-x_{0}\right) \left( y-y_{0}\right) +C\left( y-y_{0}\right) ^{2}+F=0$ using
the formula $x_{0}=\dfrac{BE-2CD}{\Delta }$, $y_{0}=\dfrac{BD-2AE}{\Delta }$%
; One can then obtain $a^{2}b^{2}$ without needing to know $a^{2}$ or $b^{2}$
since it follows easily by Lemma \ref{L1} that $a^{2}b^{2}=\dfrac{\Delta }{%
\delta }=\dfrac{4AC-B^{2}}{CD^{2}+AE^{2}-BDE-F(4AC-B^{2})}$, where $%
Ax^{2}+Bxy+Cy^{2}+Dx+Ey+F=0$ is any given equation of $E_{0}$. Multiplying
thru by $\dfrac{-a^{2}b^{2}}{F}$ then yields the proper form.
\end{remark}

\begin{proof}
As in the proof of Theorem \ref{T1}, we may assume, without loss of
generality, that $x_{0}=y_{0}=0$, so that the equation of $E_{0}$ has the
form $Ax^{2}+Bxy+Cy^{2}-a^{2}b^{2}=0$. By Lemma \ref{L2}, $\Delta \neq 0$,
and by Lemma \ref{L1}, with $D=E=0$ and $F=-a^{2}b^{2}$, we have $\delta
=a^{2}b^{2}\Delta $, which implies that $\mu =\dfrac{4a^{2}b^{2}\Delta }{%
\Delta ^{2}}=\dfrac{4a^{2}b^{2}}{\Delta }$; Also by Lemma \ref{L1}, $%
a^{2}b^{2}=\dfrac{\mu ^{2}}{4}{\large (}(A+C)^{2}-(A-C)^{2}-B^{2}{\large )}%
=\left( \dfrac{4\delta ^{2}}{\Delta ^{4}}\right) \Delta =\dfrac{4\delta ^{2}%
}{\Delta ^{3}}$; Thus $\dfrac{\delta }{\Delta }=a^{2}b^{2}=\dfrac{4\delta
^{2}}{\Delta ^{3}}$, and so $\delta =\dfrac{1}{4}\Delta ^{2}$; Hence $\mu =%
\dfrac{4\delta }{\Delta ^{2}}=1$, which implies, by Lemma \ref{L1}, that $%
a^{2}=\dfrac{A+C+r}{2}$. Substituting $a^{2}-A=\dfrac{r+C-A}{2}$ and $%
a^{2}-C=\dfrac{r+A-C}{2}$ into Theorem \ref{T1} yields Theorem \ref{T2}.
\end{proof}

\section{Rotation Angle}

Below we give a formula for the rotation angle, $\theta $, of a
non--circular ellipse, $E_{0}$, as a function of the coefficients of an
equation of $E_{0}$. We also give a simple formula for $\tan \theta $. Here
we are assuming that $0\leq \cot ^{-1}x<\pi $.

\begin{theorem}
\label{T3}Let $E_{0}$ be an ellipse with equation $%
Ax^{2}+Bxy+Cy^{2}+Dx+Ey+F=0$, with $A,C>0$; Let $\theta $\ denote the
counterclockwise angle of rotation to the major axis of $E_{0}$ from the
line thru the center of $E_{0}$ and parallel to the $x$\ axis, with $0\leq
\theta <\pi $;Let $r=\sqrt{(A-C)^{2}+B^{2}}$.

(i) $\theta =\left\{ 
\begin{array}{ll}
(1+\limfunc{sgn}B)\dfrac{\pi }{4}+\dfrac{1}{2}\cot ^{-1}\left( \dfrac{A-C}{B}%
\right) & \text{if }B\neq 0\  \\ 
{\large (}1+\limfunc{sgn}(A-C){\large )}\dfrac{\pi }{4} & \text{if }B=0%
\end{array}%
\right. $ and

(ii) $\tan \theta =\dfrac{B}{A-C-r}$ if $B\neq 0$ or $B=0$ and $A<C$.
\end{theorem}

\begin{remark}
Note that if $A=C$ and $B=0$, we have a circle and hence no rotation angle.
\end{remark}

\begin{proof}
Again, we assume that $x_{0}=y_{0}=0$, where $E_{0}$ has center $%
=(x_{0},y_{0})$; By (\ref{3}) and (\ref{4}), $A-C=y_{c}^{2}-x_{c}^{2}=c^{2}%
\sin ^{2}\theta -c^{2}\cos ^{2}\theta =(a^{2}-b^{2})(\sin ^{2}\theta -\cos
^{2}\theta )=(b^{2}-a^{2})\cos (2\theta )$, and $B=-2x_{c}y_{c}=-2c^{2}\cos
\theta \sin \theta =(b^{2}-a^{2})\sin (2\theta )$ for any $0\leq \theta \leq
\pi $; We use the well--known formula $\cot (2\theta )=\dfrac{A-C}{B}$,
which implies that $\theta =\dfrac{1}{2}\cot ^{-1}\left( \dfrac{A-C}{B}%
\right) $ or $\theta =\dfrac{\pi }{2}+\dfrac{1}{2}\cot ^{-1}\left( \dfrac{A-C%
}{B}\right) $, depending upon whether $\theta $ lies in quadrant 1 or
quadrant 2. Note that $b^{2}-a^{2}<0$;

\textbf{Case 1: }$B\neq 0$

If $B<0$, then $(b^{2}-a^{2})\sin (2\theta )<0$, which implies that $\sin
(2\theta )>0$ and so $0<\theta <\dfrac{\pi }{2}$; Thus $\theta =\dfrac{1}{2}%
\cot ^{-1}\left( \dfrac{A-C}{B}\right) $; If $B>0$, then $(b^{2}-a^{2})\sin
(2\theta )>0$, which implies that $\sin (2\theta )<0$ and so $\dfrac{\pi }{2}%
<\theta <\pi $; Thus $\theta =\dfrac{\pi }{2}+\dfrac{1}{2}\cot ^{-1}\left( 
\dfrac{A-C}{B}\right) $.

\textbf{Case 2: }$B=0$\textbf{\ }

Then $\sin (2\theta )=0$, which implies that $\theta =0$ or $\theta =\dfrac{%
\pi }{2}$; If $A>C$, then $A-C=(b^{2}-a^{2})\cos (2\theta )>0$, which
implies that $\cos (2\theta )<0$ and so $\theta =\dfrac{\pi }{2}$; If $A<C$,
then $A-C=(b^{2}-a^{2})\cos (2\theta )<0$, which implies that $\cos (2\theta
)>0$ and so $\theta =0$. That proves (i).

While one could use the fact that $\cot (2\theta )=\dfrac{A-C}{B}$, we find
it easier to proceed as follows to prove (ii). Now by (\ref{4}), $\tan
\theta =\dfrac{y_{c}}{x_{c}}$

\textbf{Case 1:} $B\neq 0$

Then by Theorem \ref{T2}, $x_{c}=\sqrt{(r+C-A)/2}$ and $y_{c}=-(\limfunc{sgn}%
B)\sqrt{(r+A-C)/2}{\large )}$, which implies that

$\tan \theta =\dfrac{-(\limfunc{sgn}B)\sqrt{(r+A-C)/2}{\large )}}{\sqrt{%
(r+C-A)/2}}=-(\limfunc{sgn}B)\sqrt{\dfrac{r+A-C}{r+C-A}}$; Now $\dfrac{r+A-C%
}{r+C-A}$ simplifies to $\dfrac{B^{2}}{{\large (}r-(A-C){\large )}^{2}}$,
and $r-(A-C)=\sqrt{(A-C)^{2}+B^{2}}-(A-C)>0$ using the inequality $\sqrt{%
x^{2}+y^{2}}-x>0$ for any $x$; Thus $\tan \theta =-(\limfunc{sgn}B)\sqrt{%
\dfrac{B^{2}}{{\large (}r-(A-C){\large )}^{2}}}=\dfrac{-(\limfunc{sgn}%
B)\left\vert B\right\vert }{r-(A-C)}=\dfrac{-B}{r-(A-C)}=\dfrac{B}{A-C-r}$;

\textbf{Case 2:} $B=0$

If $A<C$, then by Theorem \ref{T2}, $x_{c}=\sqrt{(r+C-A)/2}$ and $y_{c}=%
\sqrt{(r+A-C)/2}{\large )}$, and the rest follows as above.
\end{proof}

\section{Example}

Consider the ellipse with equation $\allowbreak 4x^{2}+2xy+6y^{2}-6x+10y=1$;
Using $A=4$, $B=2$, $C=6$, $D=-6$, and $E=10$, one has $\Delta
=4(4)(6)-2^{2}=\allowbreak 92$ and $\delta =\allowbreak 828$; Using Remark %
\ref{R1} yields $a^{2}b^{2}=\dfrac{4\delta ^{2}}{\Delta ^{3}}=\allowbreak 
\dfrac{81}{23}$ and $x_{0}=\allowbreak 1$, $y_{0}=\allowbreak -1$; Rewriting
the equation gives $4(x-1)^{2}+2(x-1)(y+1)+6(y+1)^{2}-9=0$; Multiplying thru
by $\dfrac{-a^{2}b^{2}}{-9}=\dfrac{9}{23}$ yields $\dfrac{36}{23}(x-1)^{2}+%
\dfrac{18}{23}(x-1)(y+1)+\dfrac{54}{23}(y+1)^{2}-\dfrac{81}{23}=0$; Now we
have $A=\dfrac{36}{23}$, $B=\dfrac{18}{23}$, $C=\dfrac{54}{23}$, and $%
r=\allowbreak \dfrac{18\sqrt{2}}{23}$; By Theorem \ref{T2}, $%
F_{2}=\allowbreak \left( 1+\dfrac{3}{23}\sqrt{23+23\sqrt{2}},-1-\dfrac{3}{23}%
\sqrt{-23+23\sqrt{2}}\right) $, which implies that $F_{1}=\left( 1-\dfrac{3}{%
23}\sqrt{23+23\sqrt{2}},-1+\dfrac{3}{23}\sqrt{-23+23\sqrt{2}}\right)
\allowbreak $; By Theorem \ref{T3}, $\theta =(1+\limfunc{sgn}B)\dfrac{\pi }{4%
}+\dfrac{1}{2}\cot ^{-1}\left( \dfrac{A-C}{B}\right) =\allowbreak \dfrac{%
7\pi }{8}$; One can verify that $\tan \theta =\dfrac{B}{A-C-r}=\allowbreak 1-%
\sqrt{2}\allowbreak $; \newline

\end{document}